\documentclass{amsproc}
\usepackage{amsmath}
  \usepackage{paralist}
  \usepackage{graphics} 
  \usepackage{epsfig} 
 \usepackage[colorlinks=true]{hyperref}
\usepackage{amssymb}
\hypersetup{urlcolor=blue, citecolor=red}

\newtheorem{theorem}{Theorem}[section]
\newtheorem{corollary}{Corollary}

\newtheorem{lemma}[theorem]{Lemma}
\newtheorem{proposition}{Proposition}

\theoremstyle{definition}

\newtheorem{remark}{Remark}

\title[Non-integrability criterium for normal variational equations]
{Non-integrability criterium for normal variational equations around
an integrable subsystem and an example: the Wilberforce
spring-pendulum}

\author[Acosta-Hum\'anez, Alvarez, Bl\'azquez and Delgado]{}


 \email{pacostahumanez@uninorte.edu.co}
 \email{mar@xanum.uam.mx}
 \email{david.blazquez-sanz@usa.edu.co}
 \email{jdf@xanum.uam.mx}

\begin{document}
\maketitle

\centerline{\scshape Primitivo B. Acosta-Hum\'anez}
\smallskip
{\footnotesize
 \centerline{Departamento de Matem\'aticas y Estad\'{\i}stica}
   \centerline{Universidad del Norte}
   \centerline{Barranquilla, Colombia}}
\medskip

\centerline{\scshape  Martha Alvarez-Ram\'{\i}rez }
\smallskip
{\footnotesize
\centerline{ Departamento de Matem\'aticas,}
\centerline{ UAM--Iztapalapa, 09340 Iztapalapa, }
\centerline{M\'exico, D.F., M\'exico}}
\medskip

\centerline{\scshape David Bl\'azquez-Sanz}
\smallskip
{\footnotesize
\centerline{ Universidad Sergio Arboleda,}
\centerline{Calle 74 no. 14-14,}
\centerline{ Bogot\'a, DC.,  Colombia.}}
\medskip

\centerline{\scshape Joaqu\'{\i}n Delgado}
\smallskip
{\footnotesize
\centerline{ Departamento de Matem\'aticas,}
\centerline{ UAM--Iztapalapa, 09340 Iztapalapa, }
\centerline{M\'exico, D.F., M\'exico}}

\bigskip

 \centerline{(Communicated by the associate editor name)}

\begin{abstract}
In this paper we analyze the non-integrability of the Wilbeforce
spring-pendulum by means of Morales-Ramis theory in where is enough to prove
that the Galois group of the variational equation is not virtually
abelian. We obtain these non-integrability results due to the
algebrization of the variational equation falls into a Heun
differential equation with four singularities and then we apply
Kovacic's algorithm to determine its non-integrability.
\end{abstract}

\section{Introduction}
The Wilberforce pendulum consists of a mass hanging on a flexible spiral spring that is free to oscillate in both the standard
longitudinal mode and the torsional mode.
When the mass is lifted above its equilibrium point and released from rest, it oscillates up and down along a vertical line,
slowly transferring its energy into a rotational oscillation. If the nuts screwed onto the vanes protruding from the sides of the mass are adjusted
to give the appropriate moment of inertia such that the frequencies of the longitudinal mode and the torsional mode are the same,
the pendulum will transfer its energy back and forth completely between these two modes os oscillation, see \cite{berg}.
 We shall refer to the Wilberforce spring-pendulum when the spring is free to swing in a plane, thus adding an extra
 degree of freedom. This system
 displays evidence inherent to chaotic dynamical systems.
Hence, establishing its non-integrability is our goal. The main tool for studying the non-integrability of this kinds of Hamiltonian systems is Morales-Ramis theory.
One important criterion to obtain non-integrability of autonomous Hamiltonian systems by means of differential Galois theory is due to J. J. Morales-Ruiz and J.P Ramis \cite{MoralesRamis2001a},
and Morales-Ruiz \cite{Morales1999}.  The Morales-Ramis theorem  connects two notions: integrability of Hamiltonian systems and integrability of linear systems of equations. In particular, Morales-Ramis theory has been used to prove the non-integrability of spring pendulum systems, see \cite{chdr, mpw, Morales1999}.

This paper is organized as follows. We start with a brief description of the basics
of the Morales-Ramis theory of non-integrability of the Hamiltonian system, geometric objects associated
and  the theory of linear homogeneous system of differential equation with singular points,
following \cite{BlazquezMorales} and \cite{acthesis}.
In the section \ref{wilforce} we present the necessary theoretical background to
apply Morales-Ramis theory to  the Wilberforce spring--pendulum Hamiltonian system.
Finally, the Section \ref{var-eq}  is devoted to a detailed analysis of
apply the Theorem \ref{main} in order to prove that the Wilberforce spring--pendulum Hamiltonian system is non-integrable in terms of meromorphic first integrals.

\section{Variational equations and connections}\label{sec1}
Let ${\bf M}$ be complex analytic manifold and $\vec X$ a meromorphic vector field
on ${\bf M}$. Let us denote by $K$ the field of meromorphic functions
on ${\bf M}$ and $\partial\colon K\to K$ the derivative with
respect to $\vec X$ so that $(K,\partial)$ is a differential field.
From now on we will write $f'$ for $\partial f$ with $f$ an element of $K.$
\medskip

Let $\pi\colon {\bf E}\to {\bf M}$ be a vector bundle, that we assume to be meromorphically
trivial. Let us recall that this always holds in the algebraic case due to
the vanishing of the first Galois cohomology group of the linear group, and
also holds if ${\bf M}$ is an open Riemann surface, due to Brauer theorem.
\medskip

Let $E$ be the space of meromorphic sections of $\pi$.
A \emph{linear connection in the direction of $\vec X$} is a $K^{\partial}$
linear map $\nabla\colon E\to E$ satisfying Leibniz rule
$\nabla(fe) = (\partial f)e + f\nabla(e)$. Let us consider
two vector bundles $({\bf E},\nabla)$ and $({\bf E'},\nabla')$ over
${\bf M}$ endowed with linear connections in the direction of $\vec X$. A morphism
of connections is a $K$-linear map $\phi\colon E\to E'$ such that
$\phi\circ \nabla = \nabla' \circ \phi$.

\medskip

Let $\{e_1,\ldots,e_n\}$ be a basis of $E$. The $n^2$
coefficients $a_{ij}\in K$ in the following expression are
uniquely determined.
$$\nabla(e_i) = \sum_{j=1}^n a_{ij}e_j, \quad i = 1,\dots, n$$ 
Therefore the so-called connection matrix  $A = (a_{ij}) \in {\rm Mat}(n\times n, K)$
determines uniquely $\nabla$.

\medskip

A meromorphic section $e\in E$ is called horizontal if it satisfies $\nabla(e) = 0$. Let
us write $e = \sum_{i=1}^n{x_ie_i}$, ${\bf x} = (x_1,\ldots,x_n)^t$ and $x'=\delta x$ whenever $c\in K^{\delta}$. It follows from
equation (1) that $\bf x$ satisfy the system of linear differential equations,
$${\mathbf x}' = -A^t\mathbf x.$$ 
Therefore, the choice of a linear basis of $E$ over $K$ gives a one-to-one correspondence
between linear connections in $E$ and $n\times n$  systems of linear differential equations with
coefficients in $K$.

\begin{remark}\label{galois}
Let us assume now that $\bf M$ is a Riemann surface $\Gamma$,
the differential Galois theory in this case is developed in \cite{RamisMartinet}, and can be
generalized to the case in which the field $\mathcal C$ of constants of $K$ is the field
of complex numbers $\mathbb C$, id est, the vector field $\vec X$ does not
admit meromorphic first integrals.  To a connection $\nabla$ in the direction of $\vec X$
it corresponds an algebraic group ${\rm Gal}({\bf E},\nabla)$ which is embedded into ${\rm GL}(n,\mathbb C)$
up to a conjugacy class. In the development of our application we will just make use of the following
geometric properties of the differential Galois group.
\begin{itemize}
\item[(1)] An exhaustive morphism of connections $\pi\colon ({\bf E}, \nabla) \to ({\bf\bar E}, \bar\nabla)$ induces
an exhaustive group morphism $\pi_*\colon {\rm Gal}({\bf E},\nabla)\to {\rm Gal}({\bf E}, \bar\nabla)$.
\item[(2)] Let us consider $({\bf E},\nabla) = ({\bf E_1}\oplus {\bf E_2},\nabla_1\oplus\nabla_2)$. Then the canonical projections
induce an injective morphism of algebraic groups, $$\pi_{1*}\times\pi_{2*}\colon {\rm Gal}({\bf E}, \nabla) \to
{\rm Gal}({\bf E_1}, \nabla_1)\times{\rm Gal}({\bf E_2},\nabla_2).$$
\end{itemize}
Those geometric properties and many other were exhaustively studied in the more general case of
the Lie-Vessiot systems in \cite{BlazquezMorales}. The necessary differential Galois theory background
here is given in Appendix A.
\end{remark}

\subsection{Generic variational connection}

Let us consider $\bf M$, $\vec X$, $K$ and $\partial$ as in Section 1.
There are two equivalent ways of introducing the variational equation
for $\vec X$. The vector field $\vec X$ can be naturally prolonged to
the tangent bundle ${\rm T}{\bf M}$ by application of the chain rule as done
in [1]. Equivalently we can consider the Lie bracket with $\vec X$ as a
linear connection. Let ${\mathcal T}{\bf M}$ be the space of meromorphic vector
fields in ${\bf M}$, then
$$\nabla_{\vec X}\colon {\mathcal T}{\bf M}\to {\mathcal T}{\bf M}, \quad \vec Y \mapsto \nabla_{\vec X}\vec Y = [\vec X,\vec Y]$$
is a linear connection in ${\rm T}{\bf M}$ in the direction of $\vec X$.
This is called the first variational connection of $\vec X$ on the generic point
of ${\bf M}$

\subsection{Normal variational connection and equations}

  Let $\bf N$ be an invariant manifold of $\bf M$ with respect to the flow of $\vec X$. The following exact sequence of
vector bundles over $\bf N$,
$$0 \to {\rm T}{\bf N} \to {\rm T}{\bf M}|_{\bf N} \to \mathcal {\rm Norm}({\bf N},{\bf M})\to 0$$
implicitly defines the normal bundle to $\bf N$ in $\bf M$. It is clear that the connection $\nabla_{\vec X}$
restrict to ${\rm T}{\bf M}|_{\bf N}$ and ${\rm T}{\bf N}$ and therefore we have also an
exact sequence of connections.
$$0 \to ({\rm T}{\bf N},\nabla_{\vec X|_{\bf N}})
\to ({\rm T}{\bf M}|_{\bf N},\nabla_{\vec X}|_{\bf N}) \to \mathcal ({\rm Norm}({\bf N},{\bf M}),\bar\nabla)\to 0$$
Let us consider a local system of coordinates $x_1,\ldots,x_r,y_1,\ldots,y_m$
in $\bf M$ such that the equations of $\bf N$ are written,
$$y_1=0, \quad \ldots,\quad y_m=0,$$
and those of $\vec X$,
$$\vec X = \sum_{i=1}^r f_i\frac{\partial}{\partial x_i} + \sum_{i=1}^m g_i\frac{\partial}{\partial y_i}.$$
Then, the equations of horizontal section for $\nabla$, the variational equation along $\vec X$
restricted to $\bf N$ are written,
\begin{equation}
\begin{split} 
&\delta x'_i = \sum_{j=1}^n\left.\frac{\partial f_i}{\partial x_j}\right|_N\delta x_j
+ \sum_{j=1}^m\left.\frac{\partial f_i}{\partial y_j}\right|_{N}\delta y_j \\
&\delta y'_i = 
\sum_{j=1}^m\left.\frac{\partial g_i}{\partial y_j}\right|_{N}\delta y_j
\end{split}
\end{equation}

In this system of coordinates, we can easily project the equations onto the
normal bundle obtaining the equations for horizontal sections of $\bar\nabla$,
\begin{equation}
\begin{split} 
&\delta  y'_i = \sum_{j=1}^m\left.\frac{\partial g_i}{\partial y_j}\right|_{N}\delta y_j
\end{split}
\label{NormalVariationalGeneric}
\end{equation}

Let $\Gamma$ be an invariant curve (Riemann surface) of $\vec X$ contained in $\bf N$. Its
invariance by $\vec X$ means that the connections discussed above specialize to the bundles
restricted to $\Gamma$,
\begin{equation}\label{exact}
0 \to ({\rm T}{\bf N}|_\Gamma,\nabla_{\vec X|_{\bf N}}|_{\Gamma})
\to ({\rm T}{\bf M}|_{\Gamma},\nabla_{\vec X}|_{\Gamma}) \to
({\rm Norm}({\bf N},{\bf M})|_{\Gamma},\bar\nabla|_\Gamma)\to 0
\end{equation}
The connection $\nabla_{\vec X}$ is called the variational connection of $\vec X$ along
the integral curve $\Gamma$, and $\bar\nabla|_\Gamma$ is called the variational connection
of $\vec X$ normal to ${\mathbf N}$ in ${\mathbf M}$ along $\Gamma$. The equations of horizontal sections for
this last one are written
\begin{equation}
\delta  y'_i = \sum_{j=1}^m\left(\vec X|_\Gamma\tau
\right)^{-1}\left.\frac{\partial g_i}{\partial y_j}\right|_{\Gamma}\delta y_j
\label{NormalVariational}
\end{equation}
a square system of $m$ linear differential equations with meromorphic coefficients in $\Gamma$,
or equivalently a meromorphic linear connection in ${\rm Norm}({\mathbf N}, {\mathbf M})|_{\Gamma}$. This is the so called
\emph{variational equations of the flow of $\vec X$
normal to the invariant manifold ${\bf N}$ along the integral curve $\Gamma$.}

\begin{remark}
It is important to remark that the invariant curve $\Gamma$ may not corresponds to a unique
trajectory of the flow of $\vec X$ but also contain equilibrium points. This is important because
equilibrium points in $\Gamma$ are poles of $(\vec X|_\Gamma\tau)^{-1}$ and
therefore singularities of the system (\ref{NormalVariational}). In many applications
we also replace the original manifold ${\bf M}$ by some suitable compactification such that the considered
invariant curve $\Gamma$ has some equilibrium points at the infinity (see \cite{Morales1999} pp. 70--74).
The curve $\Gamma$ may then contain singular points like nodes or cusps. This is coherent with
the computations above, providing that the field of meromorphic functions on a singular Riemann
surface coincides with the one of its desingularization
\end{remark}

\section{Non-integrability criterium}

\subsection{Morales-Ramis theory}

From now on let $\bf M$ be a complex symplectic manifold, of dimension $2n$
with simplectic form $\Omega$. Let us consider a symplectic system
of coordinates $x_1,\ldots,x_n$, $y_1,\ldots,y_n$. If $H$ is an holomorphic
function in $\bf M$ then the Hamiltonian vector field $\vec X_H$ is defined
intrinsically by the formula,
$$dH + i_{\vec X_H}\Omega = 0,$$
and its expression in local coordinates is,
$$\frac{{\rm d} x_i}{{\rm d} t} = \frac{\partial H}{\partial y_i}, \quad \frac{{\rm d} y_i}{{\rm d} t} = - \frac{\partial H}{\partial x_i}.$$

The symplectic structure induces a \emph{Poisson bracket} for functions defined on $M$, namely
 $\{F, G\} = \vec X_FG$. The functions $F,$ $G$ in $M$ are
said to be in involution if $\{F,G\} = 0$. A Hamiltonian system is called \emph{completely
integrable} by meromorphic functions if it admits $n$ independent meromorphic first
integrals in involution. Let $\Gamma$ be an invariant curve for $\vec X_H$.
Let us define the variational connection $\nabla_{\vec X_H}$ in direction of $\vec X$ as in Section 1.
Then , $({\rm T}{\bf M}|_{\Gamma}, \nabla_{\vec X_H}|_\Gamma)$ is the variational conection
to $\vec X$ along the integral curve $\Gamma$. The equation of horizontal sections, so called
the first variational equation along $\Gamma$ is written in symplectic coordinates,
\begin{equation}\label{VariationalH}
\begin{split}
&\delta x'_i =
\sum_{j=1}^n \left(
\left.\frac{\partial^2 H}{\partial {x_j}\partial {y_i}}\right|_\Gamma \delta x_j +
\left.\frac{\partial H}{\partial {y_j}\partial {y_i}}\right|_\Gamma \delta y_j \right) \\
&\delta y'_i =
- \sum_{j=1}^n \left(
\left. -\frac{\partial^2 H}{\partial {x_j}\partial {x_i}}\right|_\Gamma \delta x_j -
\left.\frac{\partial H}{\partial {y_j}\partial {x_i}}\right|_\Gamma \delta y_j
\right)
\end{split}
\end{equation}

The Morales-Ramis theory,
relates the Liouville integrability of Hamiltonian system and
the Galois groups of differential equations. The following result
(Morales-Ramis  \cite{MoralesRamis2001a}) and subsequent generalizations
are some of the most effective and useful known for the theoretical
proof of non-integrability of Hamiltonian systems. Let us recall that
an algebraic group $G$ is said to be \emph{virtually abelian} it its
connected component of the identity $G^0$ is abelian (see Appendix A).

\begin{theorem}\label{MoralesRamis}
Let $H$ be a Hamiltonian, and $\vec X_H$ its associated Hamiltonian vector field
in ${\mathbf M}$. Let $\Gamma$ be an invariant curve for $H$.
Assume that $H$ is completely integrable by meromorphic
functions independent in a neighborhood of $\Gamma$ but not necessarily
on $\Gamma$ itself. Then the group
${\rm Gal}({\rm T}{\mathbf M}|_\Gamma,\nabla_{\vec X_H}|_\Gamma)$  is virtually abelian.
\end{theorem}

\subsection{Analysis of normal variational connection}

From now on let $\bf N$ be a symplectic submanifold of
${\mathbf M}$, i.e., the pair $({\mathbf N},\Omega|_{\mathbf N})$ is a symplectic manifold.
Let $2r$ be the dimension of $\mathbf{N}$ and $n = r + m$.

\begin{lemma}
There is a natural decomposition of the bundle ${\rm T}{\mathbf M}|_{\mathbf N}\to{\mathbf N}$ as the
direct sum of the tangent bundle to $\bf N$ and the normal bundle to $\bf N$ in $\bf M$,
$${\rm T}{\mathbf M}|_{\mathbf N} = {\rm T}{\mathbf M} \oplus {\rm Norm}({\mathbf N},{\mathbf M}).$$
\end{lemma}

\begin{proof}
Let us consider ${\rm T}{\bf N}^\bot \subset {\rm T}{\mathbf M}$ the bundle of vectors
orthogonal to ${\rm T}{\mathbf N}$ with respect to the symplectic form $\Omega$. By
hypothesis $\Omega|_N$ is non-degenerated and therefore ${\rm T}{\mathbf N}^\bot$ is
a supplementary bundle for ${\rm T}{\bf N}$. The exact sequence \eqref{exact}
identifies ${\rm T}{\mathbf N}^\bot$ with ${\mathrm Norm}({\bf N},{\mathbf M})$ and we get the
result.
\end{proof}

\begin{lemma}\label{directsum}
The variational connection $\nabla_{\vec X_H}|_{\mathbf N}$ in ${\rm T}{\mathbf M}|_{\mathbf N}$
splits as direct sum of the variational connection of  $\nabla_{\vec X_H|_{\mathbf N}}$
in ${\rm T}{\bf N}$ and the normal variational connection $\bar\nabla$ in
${\rm Norm}({\mathbf N},{\mathbf M})$,
$$\nabla_{\vec X_H}|_{\mathbf N} = \nabla_{\vec X_H|_{\mathbf N}} \oplus \bar\nabla.$$
\end{lemma}

\begin{proof}
It suffices to proof  that ${\rm T}{\bf N}^\bot$ is an invariant bundle
for $\nabla_{\vec X_H}|_{\bf N}$. Let us take $\vec Y$ meromorphic section
of ${\rm T}{\bf N}^\bot$ and $\vec Z$ meromorphic vector field in $\bf M$.
Let us note that $\mathcal L_{\vec X}\Omega = \Omega$ for any Hamiltonian vector
field, and that $[\vec X, \vec Z]$ is defined as vector
field tangent to $\bf N$. Therefore,
$$\Omega(\nabla_{\vec X}|_{\bf N}\vec Y,\vec Z) = \Omega([\vec X,\vec Y],\vec Z)
= (\mathcal L_{\vec X}\Omega)(\vec Y, \vec Z) - \Omega(\vec Y, [\vec X,\vec Z])
= 0$$
and $\nabla_{\vec X}|_{\mathbf N}\vec Y$ is orthogonal to ${\rm T}{\bf N}$ with respect
to $\Omega$.
\end{proof}

\begin{theorem}\label{main}
Let $H$ be a Hamiltonian system in ${\mathbf M}$. Assume that ${\mathbf N}$ is invariant by the flow
of $\vec X_H$. Assume that $\vec X_H|_N$ is completely integrable, as Hamiltonian
system in ${\mathbf N}$, by meromorphic functions. Let $\Gamma$ be an invariant curve of
$\vec X_H$ in ${\mathbf N}$. Then the differential Galois group of the variational equation
to the flow of $\vec X_H$ along $\Gamma$, ${\rm Gal}(\nabla_{\vec X}|_\Gamma)$ is virtually
abelian if and only if the differential Galois group of the variational equation to the
flow of $\vec X_H$ normal to ${\mathbf N}$ along $\Gamma$, ${\rm Gal}(\bar\nabla|_\Gamma)$,
is virtually abelian.
\end{theorem}

\begin{proof}
First, let us assume that ${\rm Gal}(\nabla_{\vec X}|_{\Gamma})$ is virtually
abelian. We consider the exact sequence \eqref{exact}. By Remark \ref{galois}  there
is exhaustive group morphism,
$$\pi_2\colon{\rm Gal}(\nabla_{\vec X}|_{\Gamma})\to {\rm Gal}(\bar\nabla|_{\Gamma}).$$
Therefore, if ${\rm Gal}(\nabla_{\vec X}|_{\Gamma})$ is virtually abelian then its quotient
${\rm Gal}(\bar\nabla|_{\Gamma})$ is also abelian.
Second, let us assume that ${\rm Gal}(\bar\nabla|_{\Gamma})$ is virtually
abelian. From Lemma \ref{directsum} we have
$\nabla_{\vec X}|_{\Gamma} = \nabla_{\vec X|_{\bf N}}|_{\Gamma}\oplus \bar\nabla|_{\Gamma}$,
and by Remark \ref{galois} (2) there is an injective group morphism,
$$(\pi_1\times\pi_2)\colon  {\rm Gal}(\nabla_{\vec X}|_{\Gamma})\to
{\rm Gal}(\nabla_{\vec X|_{\bf N}}|_{\Gamma}) \times
{\rm Gal}(\bar\nabla|_{\Gamma}).$$

Assume that ${\rm Gal}(\bar\nabla|_{\Gamma})$ is virtually abelian. From Theorem \ref{MoralesRamis} we
have that \\ ${\rm Gal}(\nabla_{\vec X|_{\bf N}}|_{\Gamma})$ is virtually abelian.
It follows that ${\rm Gal}(\nabla_{\vec X})$ is virtually abelian.
\end{proof}

The following corollary directly follows from Theorem \ref{MoralesRamis}.
Therefore, we can ask what is the use of Theorem \ref{main}. There
are examples, e.g. \cite{MoralesSimoSimon}, of application of the Morales-Ramis theorem
in which the non-integrability is not seen in the normal variational equation
but in the total variational equation. Theorem \ref{main} describes
a theoretical situation in which the non-integrability,
if captured by the first order variational equation then it is captured with
the normal variational equation.

\begin{corollary}
Let $H$ be a Hamiltonian system in ${\mathbf M}$. Assume that ${\mathbf N}$ is invariant by the flow
of $\vec X_H$. Assume that $\vec X_H|_N$ is completely integrable, as Hamiltonian
system in ${\mathbf N}$, by meromorphic functions. Then if $\vec X_H$ is completely integrable
in ${\mathbf M}$ by meromorphic functions then the differential Galois group of the variational
equation to the flow of $\vec X_H$ normal to ${\mathbf N}$ along $\Gamma$, ${\rm Gal}(\bar\nabla|_\Gamma)$,
is virtually abelian.
\end{corollary}

\section{The Wilbeforce-spring pendulum}\label{wilforce}
The classical Wilbeforce spring--pendulum consists of a solid cylinder
attached to a spring. The cylinder is free to perform oscillations
along  the vertical. At the same time when the spring is elongated
it produces a torque proportional to the angle of torsion of the
cylinder $\phi$ (see Figure 1). We add an extra degree of freedom by allowing the spring to swing on
a  fixed plane in space containing the vertical motions. Let $r$, $\theta$ denote polar coordinates, where $r$ is
 the distance from the fixed point  of the spring to center o mass of the cylinder and $\theta$ the angular deviation from the vertical axis.
The Lagrangian is
\begin{equation}\label{lagra}
\frac{1}{2} m
   \left(\dot{r}^2+r^2\dot{\theta}^2 \right)+
   \frac{J \dot{\phi}^2}{2}-\frac{1}{2} k
   (r-r_0)^2-\frac{\lambda  \phi ^2}{2}-\frac{1}{2}
   (r-\ell) \epsilon  \phi +g m r \cos \theta
\end{equation}
where $k$ is the constant of the spring, $r_0$ is the unstretched length of the spring, $\lambda$ gives the constant of proportionality of the angular restoring and $\epsilon$ is the coupling among the vertical
and torsional modes. The mass of the cylinder is $m$ and $J$ its
moment of inertia.
\begin{figure}[htb!]
\centering
\includegraphics[height =2.5in]{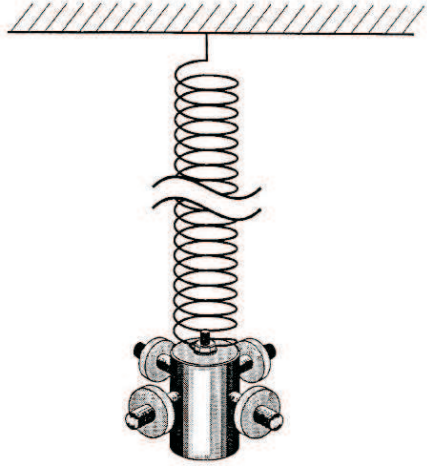}
\caption{The Wilberforce spring-pendulum. Taken from reference \cite{berg}.}
\label{fig1}
\end{figure}

\begin{remark}
Observe that for $\epsilon=0$ we recover the spring--pendulum system whose aspects of non--integrability have been studied among others by \cite{chdr}, \cite{mpw}, \cite{Morales1999} and \cite{regina}.
\end{remark}

\begin{remark} The invariant subsystem $\theta=0$, $\dot{\theta}=0$ corresponds to the classical Wilberforce pendulum which is a quadratic Lagrangian thus a \emph{linear system}.
\end{remark}

\subsection{Adimensionalization}
It will be convenient to reduce the seven parameters $m, J, k, r_0, \epsilon, g$
to four by the introduction to variables a follows:
Let $\ell=r_0+mg/k$, $\omega_p^2=g/\ell$, $\omega_s^2=k/m$. Observe that
$$
1=\frac{r_0}{\ell}+\frac{mg}{k\ell}=\frac{r_0}{\ell}+\left(\frac{\omega_p}{\omega_s}\right)^2
= \frac{r_0}{\ell}+f
$$
where the parameter
$$
f=\left(\frac{\omega_p}{\omega_s}\right)^2
$$ is the ratio of frequencies of the linear pendulum and the spring.
Then if $r=\ell\rho$,
$$
r-r_0=\ell(\rho-1+f).
$$
Also we introduce a new time variable $\tau$ via  $\frac{d\tau}{dt}=\omega_s$,
with this substitution the Lagrangian (\ref{lagra}) becomes
$$
L=\frac{1}{2}k\ell^2(\rho'^2+\rho^2\theta'^2)+\frac{Jk}{2m}\phi'^2
-\frac{1}{2}k\ell^2(\rho-1+f)^2-\frac{\lambda}{2}\phi^2 -
\frac{1}{2}\epsilon\ell (\rho - 1 )\phi
+gm\ell \rho\cos\theta
$$
a further factorization yields
$$\displaylines{ L=k\ell^2\left[
\frac{1}{2}(\rho'^2+\rho^2\theta'^2)+\frac{1}{2}\frac{J}{m\ell^2}\phi'^2
-\frac{1}{2}(\rho-1+f)^2-\frac{1}{2}\frac{\lambda}{k\ell^2}\phi^2 \right.\cr \quad \qquad
- \frac{1}{2}\frac{\epsilon}{k\ell} (\rho - 1)\phi
\left.+\frac{gm}{k\ell} \rho\cos\theta\right]\cr}
$$
where prime ($'$) indicates differentiation with respect $\tau .$

One checks easily that the constant factor can be dismissed leading to the adimensional Lagrangian,
$$
 L=
\frac{1}{2}(\rho'^2+\rho^2\theta'^2)+\frac{1}{2}\frac{J}{m\ell^2}\phi'^2
-\frac{1}{2}(\rho-1+f)^2-\frac{1}{2}\frac{\lambda}{k\ell^2}\phi^2 -
\frac{1}{2}\frac{\epsilon}{k\ell} (\rho-1)\phi
+\frac{gm}{k\ell} \rho\cos\theta.
$$
Let us introduce the following (dimensionless) parameters
$$
a=\frac{J}{m\ell^2},\quad b=\frac{\lambda}{k\ell^2},\quad c=\frac{\epsilon}{k\ell}
$$
then the Lagrangian finally becomes
\begin{equation}\label{adim}
     L=
\frac{1}{2}(\rho'^2+\rho^2\theta'^2)+\frac{1}{2}a\phi'^2
-\frac{1}{2}(\rho-1+f)^2-\frac{1}{2}b\phi^2 -
\frac{1}{2}c (\rho-1)\phi
+f \rho\cos\theta.
\end{equation}

The Hamiltonian corresponding to this coordinates assume the form
\begin{equation}
     H=
\frac{1}{2}\left(P_\rho^2+\frac{P_\theta^2}{\rho^2}\right)+\frac{1}{2a}P_\phi^2
+\frac{1}{2}(\rho-1+f)^2+\frac{1}{2}b\phi^2 +
\frac{1}{2}c (\rho-1)\phi
-f \rho\cos\theta.
\end{equation}
In fact by scaling the angle $\phi$ and its conjugate momenta we can suppose in
what follows that $a=1$ resulting with three parameters: $b,c,f$
\begin{equation}\label{hamiltonian}
     H=
\frac{1}{2}\left(P_\rho^2+\frac{P_\theta^2}{\rho^2}+P_\phi^2\right)+
\frac{1}{2}(\rho-1+f)^2+\frac{1}{2}b\phi^2 +
\frac{1}{2}c (\rho-1)\phi
-f \rho\cos\theta .
\end{equation}

\subsection{The Wilberforce linear subsystem}
The equations of motion are given by
\begin{equation}\label{eq-adim}
\begin{array}{lll}
  \rho' &=& P_\rho, \\ && \\
  \theta' &=& \displaystyle{\frac{P_\theta}{\rho^2}}, \\ && \\
  \phi' &=& P_\phi, \\ && \\
  P_\rho' &=& \displaystyle{\frac{P_\theta}{\rho^3}-(\rho-1+f)-\frac{c}{2}\phi+f\cos\theta}, \\ && \\
  P_\theta' &=& -f\rho\sin\theta, \\ && \\
  P_\phi' &=& \displaystyle{-b\phi-\frac{c}{2}(\rho-1)}.
\end{array}
\end{equation}

From these equations, one checks easily that $\theta=P_\theta=0$ is an invariant subsystem, as asserted previously. The systems reduces in this case to the linear system
\begin{equation}\label{eq-adim2}
\begin{array}{lll}
  \rho' &=& P_\rho, \\ && \\
  \phi' &=& P_\phi, \\ && \\
  P_\rho' &=& \displaystyle{-\rho+1-\frac{c}{2}\phi}, \\ && \\
  P_\phi' &=& \displaystyle{-b\phi-\frac{c}{2}(\rho-1}).
\end{array}
\end{equation}
That can be recast as   two coupled oscillators
\begin{align}
\rho'' + \rho =  & 1 - \frac{c}{2}\phi,  \label{sol-var}\\
\phi '' + b\phi = & -\frac{c}{2} \rho +\frac{c}{2}(1-f)\nonumber
\end{align}
with  Hamiltonian (set $P_{\theta}=\theta=0$ in (\ref{hamiltonian}))
\begin{equation}\label{hamiltonian2}
H_{0,0}=
\frac{1}{2}\left(P_\rho^2+P_\phi^2\right)+
\frac{1}{2}(\rho-1+f)^2+\frac{1}{2}b\phi^2 +
\frac{1}{2}c (\rho-1)\phi -f \rho.
\end{equation}
Let us exhibit the reduced system (\ref{sol-var}) as an integrable Hamiltonian system in a convenient normal form. For this purpose let
 $z=\rho-1+f$, so  (\ref{hamiltonian2}) becomes
\begin{equation}\label{hamiltonian3}
H_{0,0} (z,\phi) = \frac{1}{2}(P_z^2 + P_\phi^2) + \frac{1}{2}z^2 + \frac{b}{2}\phi^2+
\frac{1}{2}c (z-f)\phi  -fz-f(1-f).
\end{equation}
where of course the constant term can be neglected.
The following proposition makes explicit the integrals of motion in this case.

\begin{proposition}
There exists a symplectic change of coordinates that transforms the Hamiltonian
(\ref{hamiltonian3}) to the form
\begin{equation}\label{normal-form}
    H_{0,0}=P_{x_1}^2+P_{x_2}^2+\omega_1^2x_1^2+
\omega_2^2x_2^2
\end{equation}
\end{proposition}
\begin{proof}
Perform the linear symplectic change of variables with parameter~$\alpha$
\begin{eqnarray*}
z &=& x_1\cos\alpha - x_2\sin\alpha, \\
\phi &=& x_1\sin\alpha + x_2\cos\alpha,\\
P_{x_1} &=& P_{z}\cos\alpha + P_{\phi}\sin\alpha, \\
P_{x_2} &=& -P_{z}\sin\alpha + P_{\phi}\cos\alpha.
\end{eqnarray*}
Then naturally $P_{x_1}^2+P_{x_2}^2=P_z^2+P_{\phi}^2$ for any  $\alpha$.
Imposing the condition that the mixed terms $x_1 x_2$ vanish we get the choice
$$
\tan\alpha = \frac{-1+b\pm \sqrt{(b-1)^2+c^2}}{c}.
$$
Then the Hamiltonian becomes
$$
 H_{0,0}=\frac{1}{2}\left(P_{x_1}^2+P_{x_2}^2+\omega_1^2x_1^2+\omega_2^2x_2^2\right)
$$
where
\begin{eqnarray}
\omega_1^2 &= \cos^2\alpha + b\sin^2\alpha + c \cos\alpha\sin\alpha =\frac{1}{2}\left(1+b+\sqrt{(-1+b)^2+c^2}\right)\label{omega1}\\
\omega_2^2 &= \sin^2\alpha + b\cos^2\alpha - c\cos\alpha\sin\alpha =\frac{1}{2}\left(1+b-\sqrt{(-1+b)^2+c^2}\right)\label{omega2}
\end{eqnarray}
\end{proof}

\subsection{The normal variational equations}
We next calculate the full variational equations of (\ref{eq-adim}) along of the solution of (\ref{sol-var}), which  can be written as
\begin{equation}\label{var1}
\begin{array}{lll}
 \delta\rho' &=& \delta P_\rho, \\ && \\
 \delta\theta' &=& \displaystyle{\frac{\delta P_\theta}{\rho^2} -\frac{2  P_\theta\delta\rho}{\rho^3}},\\ && \\
 \delta\phi' &=& \delta P_\phi\\ && \\
 \delta P'_\rho &=& \displaystyle{\frac{\delta P_\theta}{\rho^3}- 3 \frac{P_\theta}{\rho^4} \delta\rho - \delta\rho -\frac{c}{2}\delta\phi - f\delta\theta\sin\theta},\\ && \\
 \delta P'_\theta &=& -f\delta\rho\sin\theta -f\rho \:\delta\theta\cos\theta,\\ && \\
 \delta P'_\phi &=&  - \displaystyle{\frac{c}{2}\delta\rho} -b\delta\phi.
\end{array}
\end{equation}

It follows immediately from equations (\ref{var1}) that the {\em tangencial} equations
along the invariant plane $\theta  = P_\theta = 0$ are
\begin{equation*}
\begin{array}{lll}
 \delta\rho' &=& \delta P_\rho, \\ && \\
 \delta\phi' &=& \delta P_\phi,\\ && \\
 \delta P'_\rho &=& \displaystyle{ -\delta\rho -\frac{c}{2}\delta\phi },\\ && \\
 \delta P'_\phi &=& - \displaystyle{\frac{c}{2}\delta\rho} -b\delta\phi.
\end{array}
\end{equation*}
Then this can be rewritten as
\begin{align*}
\left( \begin{array}{c} \delta \rho' \\ \delta\phi \\ \delta P_\rho' \\ \delta P_\phi'\end{array} \right) &= \left( \begin{array}{llll}
0 & 0 & 1 & 0 \\
0 & 0 & 0 & 1 \\
-1 & -\frac{c}{2} & 0 & 0 \\
-\frac{c}{2} & -b & 0 & 0
        \end{array}\right)
     \left( \begin{array}{c} \delta \rho \\ \delta\phi \\ \delta P_\rho \\ \delta P_\phi\end{array} \right)
\end{align*}
The $2\times 2$ lower matrix has precisely the eigenvalues $-\omega_{1,2}^2$  given in (\ref{omega1}), (\ref{omega2}).
Thus the eigenvalues of the full matrix are $\pm i\omega_1$, $\pm i\omega_2$ whenever $\omega_{1,2}^2\neq 0$. The case  $\omega_{1,2}^2>0$ is obtained whenever $c^2<4b$, and the general solution of this system is
\begin{align*}
\rho (t) &= A\sin \omega_1 t + B\cos \omega_2 t, \\
\phi (t) &= C\sin \omega_1 t + D\cos \omega_2 t
\end{align*}
with arbitrary constants $A$, $B$, $C$ and $D$. In the same way, the case $\omega_{1,2}^2<0$ is obtained whenever $c^2>4b$, and the general solution of this system is
\begin{align*}
\rho (t) &= A_2\sinh \omega_1 t + B_2\cosh \omega_2 t, \\
\phi (t) &= C_2\sinh \omega_1 t + D_2\cosh \omega_2 t
\end{align*}
with arbitrary constants $A_2$, $B_2$, $C_2$ and $D_2$.

In general, for $\omega_{1,2}^2\neq 0$, provided whenever $c^2\neq 4b$, the general solution of such system is
\begin{align*}
\rho (t) &= A_3e^{\omega_1 t} + B_3e^{\omega_2 t}, \\
\phi (t) &= C_3e^{\omega_1 t} + D_3e^{\omega_2 t}
\end{align*}
with arbitrary constants $A_3$, $B_3$, $C_3$ and $D_3$.

On the invariant plane the {\em normal} variational equation becomes
\begin{align*}
 \delta\theta' &= \displaystyle{\frac{1}{\rho^2}\delta P_\theta },\\
 \delta P'_\theta &=  -f\rho  \:\delta\theta.
\end{align*}

Then we may rewrite this systems as a second-order equation of the form

\begin{equation}\label{evgen} \delta\theta'' + \displaystyle{\frac{2\rho'}{\rho}\delta \theta' } - \frac{f}{\rho} \delta\theta =0.
\end{equation}

\section{Galoisian analysis of the variational equation}\label{var-eq}
Considering the solution $\rho(t)=A\sin \omega_1 t + B\cos \omega_2 t$, the variational equation \eqref{evgen} becomes to

\begin{equation}\label{paheq1}
\partial_t^2\theta(t) +{2A\omega_1\cos(\omega_1t) -2B\omega_2\sin(\omega_2t)\over
A\sin(\omega_1t) +B\cos(\omega_2t)}\partial_t\theta(t)-{f\over A\sin(\omega_1t)
+B\cos(\omega_2t)}\theta(t),
\end{equation}
where $\partial_t=\dot{}=\frac{d}{dt}$ (see appendices). To avoid
triviality we assume $f\neq 0$, $\omega_1\neq 0$ and $\omega_2\neq 0$. By means
of the change of dependent variable $\eta=(A\sin \omega_1t+B\cos
\omega_2t)\theta$ this equation is transformed in the reduced form:
\begin{equation}\label{paheq2}
\partial^2_t\eta(t) = \left({f-A\omega_1^2\sin(\omega_1t)-B\omega_2^2\cos(\omega_2t) \over
A\sin(\omega_1t) +B\cos(\omega_2t)} \right) \eta(t).
\end{equation}
The equations \eqref{paheq1} and \eqref{paheq2} cannot be algebrized
using Hamiltonian Algebrization (see Appendix C) due to in general
there is not exists $r\in \mathbb{Q}$ such that $\omega_2=r\omega_1$, see
\cite{acthesis,acmowe} and references therein. To apply Morales-Ramis
theory, and in particular our result (Theorem 2.) we need one
particular solution for the variational equation in where  its Galois
group is not virtually abelian. Thus, we can assume $A=0$ and the
equation \eqref{paheq2} is reduced to
\begin{equation}\label{paheq3}
\partial_t^2\eta(t)= \left(\frac {f}{B\cos(\omega_2t)}-\omega_2^2 \right) \eta(t).
\end{equation}
On the other hand, if we assume $B=0$, we obtain
\begin{equation}\label{paheq4}
\partial_t^2\eta \left( t \right) = \left({\frac {f}{A\sin \left(
\omega_{1}t \right) }-\omega_1^2} \right) \eta \left( t
\right).
\end{equation}
Now, we proceed to apply the Hamiltonian algebrization for
\eqref{paheq3} considering $B\neq 0$. We can see that $x=x(t)=\cos(\omega_2t)$ is a Hamiltonian
change of variable, where $\sqrt{\alpha}=-\omega_2\sqrt{1-x^2}$. Therefore,
the algebrization of \eqref{paheq3}  is
\begin{equation}\label{paheq5}
\widehat{\partial}_x^2\widehat{\eta}(x) =
\left(\frac{\lambda}{x}-\omega_2^2 \right)\widehat{\eta}(x),\quad
\widehat{\partial}_x=\sqrt{\alpha}\partial_x,\quad
\lambda=\frac{f}{B},\quad \widehat{\eta}(x(t))=\eta(t).
\end{equation}
Thus, the equation \eqref{paheq5} is equivalent to
\begin{equation}\label{paheq6}
\partial_x^2\widehat{\eta}(x)+{x\over
x^2-1}\partial_x\widehat{\eta}(x)+ {\lambda-\omega_2^2x\over
\omega_2^2x(x^2-1)}\widehat{\eta}(x)=0.
\end{equation}
The equation \eqref{paheq6} has four singularities: $0,1,-1,\infty$
which are of regular type. This means that this equation can be
transformed into a \textit{Heun equation}:
$$\partial_x^2y(x) + \left(\frac {\gamma}{x
}+\frac {\delta}{x-1}+\frac {\epsilon}{x-a} \right)\partial_x y(x)
+\frac {\mu\beta x-q }{x \left( x-1 \right)
\left( x-a \right)} y(x)=0,
$$
where $\epsilon+\gamma+\delta-\mu-\beta=1$. With the translation
$x\mapsto x-1$ we obtain the Heun equation
\begin{equation}\label{paheq7}
\partial_x^2\widehat{\eta}(x) + \left(\frac {1}{2x
}+\frac {1}{2x-4} \right)\partial_x \widehat{\eta}(x) +\frac
{-\omega_2^2x+\lambda+1 }{\omega_2^2x \left( x-1 \right)
\left( x-2 \right)} \widehat{\eta}(x)=0,
\end{equation} with parameters
$$a=2,\quad\epsilon=\gamma=\frac{1}{2},\quad \delta=0,\quad
\mu=1,\quad \beta=-1,\quad q=-{\lambda\over \omega_2^2}-1.$$
Through the change of variable $\xi=\sqrt[4]{x(x-2)}\widehat{\eta}$
we obtain the reduced form of the equation \eqref{paheq7} to apply the
Kovacic algorithm:
\begin{equation}\label{paheq8}
\partial_x^2\xi=r\xi,\quad r={\frac {3\,{\omega_2}^{2}{x}^{3}-\left(4\,\lambda+9\,{
\omega_2}^{2} \right){x}^{2} +\left( 8\,\lambda+3\,{\omega_2}^{2} \right)x +3\,{\omega_2}^{2}}{{4\omega_2}^{2}{x}^{2} \left( x-1 \right)  \left( x-2
\right) ^{2}}}.
\end{equation}
For our purposes, in order to apply the Kovacic algorithm (see
Appendix B and see also \cite{ko}), we can write $r$ as follows:
\begin{equation}\label{paheq8bis}
r=\frac{- 3}{16(x - 2)^2} + {9\omega_2^2 - 8\lambda\over 16\omega_2^2(x - 2)} +
{\lambda\over \omega_2^2(x - 1)} - \frac{3}{16x^2} - {8\lambda +
9\omega_2^2\over 16\omega_2^2x}.
\end{equation}
Now we start the analysis of the equation \eqref{paheq8} applying
Kovacic's algorithm. We can see that $\Gamma=\{0,1,2,\infty\}$ and
$\circ (r_0)=\circ (r_2)=\circ (r_\infty) =2$ and $\circ (r_1)=1$. By
case 1, step 1, we fall into the conditions $(c_1)$ for $c=1$, $(c_2)$
for $c=0,2$ and $(\infty_2)$ for $\infty$.
In this way we obtain
\begin{align*}
& \sqrt{r}_0=\sqrt{r}_1=\sqrt{r}_2=\sqrt{r}_\infty=0,\qquad
\alpha^+_0=\alpha^+_2=\frac34,\\
&\alpha^-_0=\alpha^-_2=\frac14,\qquad \alpha^+_1=\alpha^-_1=1,\qquad \alpha^+_\infty=\frac{3}{2},\qquad \alpha^-_\infty=-\frac{1}{2}.
\end{align*}
By step
2 we obtain $D=\{0\}$ and by step 3 we obtain that $P_0=1$ does not
satisfy the relation \eqref{recu1}. This means, for differential
Galois theory (see Appendix A), that the Galois group of the
variational equation is not a subgroup of the Borel group (except for
$\lambda=0$ or $\omega_2=0$, avoided from the start). We follow with case 2
of Kovacic's algorithm in where we fall into the conditions $(c_1)$
for $c=1$, $(c_2)$ for $c=0,2$ and $(\infty_2)$ for $\infty$.
In this way, by step 1, we obtain $$E_0=E_2=\{1,2,3\},\quad E_1=\{4\},\quad
E_\infty=\{-2,2,6\}.$$
Now, by step 2, we obtain $D=\{0\}$ and by step
3 we obtain that $P_0=1$ does not satisfy the relation \eqref{recu2}.
This means, again for differential Galois theory, that the Galois
group of the variational equation is not conjugated to a subgroup of
the infinite Dihedral group. Finally we look into the case 3 of
Kovacic's algorithm. We first consider $m=4$, thus, by step 1, we fall into the conditions
$(c_1)$ for $c=1$, $(c_2)$ for $c=0,2$ and $(\infty)$ for $\infty$.
In this way we obtain $$E_0=E_2=\left\{3,6,9\right\},\quad E_1=\{12\},\quad E_\infty=\{\pm 6,0,12,18\}.$$
By step 2 we obtain $D=\{0\}$ and for instance the rational function $\theta$ and the polynomial $S$ are given by $$\frac{6x^2-12x+2}{x^3+3x^2+2x},\quad S=x^3+3x^2+2x.$$ By step 3, we have the monic polynomial of degree $0$, that is $P=1$, and the sequence of polynomials is obtained by means of the relation (\ref{recu3}):
$$\begin{array}{l}
P_4=-1,\\
P_3=-6x^2+12x-2,\\
P_2=-6x^4+\cdots+3,\\
P_1=24x^6+\cdots+3,\\
P_0=-216x^8+\cdot-\frac{51}{2},\\
P_{-1}=3024x^{10}+\cdot+144 \textit{  non-identically zero!}
\end{array}$$ and for instance the Galois Group is not the tetrahedral group. Now we consider $m=6$, thus, by step 1, we fall into the conditions
$(c_1)$ for $c=1$, $(c_2)$ for $c=0,2$ and $(\infty)$ for $\infty$.
In this way we obtain $$E_0=E_2=\left\{3,4,5,6,7,8,9\right\},\quad E_1=\{12\},\quad E_\infty=\{\pm2,\pm6,10,14,18\}.$$
By step 2 we obtain $D=\{0\}$ and for instance the rational function $\theta$ and the polynomial $S$ are given by $$\frac{9x^2-18x+3}{x^3+3x^2+2x},\quad S=x^3+3x^2+2x.$$ By step 3, we have the monic polynomial of degree $0$, that is $P=1$, and the sequence of polynomials is obtained by means of the relation (\ref{recu3}):
$$\begin{array}{l}
P_6=-1,\\
P_5=-9x^2+18x-3,\\
\vdots\\
P_0=-38016x^{12}+\cdots-\frac{3123}{4},\\
P_{-1}=798336x^{14}+\cdots+6156 \textit{  non-identically zero!}
\end{array}$$ and for instance the Galois Group is not the  octahedral group. Finally we consider $m=12$, thus, by step 1, we fall into the conditions
$(c_1)$ for $c=1$, $(c_2)$ for $c=0,2$ and $(\infty)$ for $\infty$.
In this way we obtain $$E_0=E_2=\left\{3,4,5,6,7,8,9\right\},\quad E_1=\{12\},\quad E_\infty=\{0,\pm 2,\pm 4,\pm6, 8,10,12,14,16,18\}.$$By step 2 we obtain $D=\{0\}$ and for instance the rational function $\theta$ and the polynomial $S$ are given by $$\frac{18x^2-36x+6}{x^3+3x^2+2x},\quad S=x^3+3x^2+2x.$$ By step 3, we have the monic polynomial of degree $0$, that is $P=1$, and the sequence of polynomials is obtained by means of the relation (\ref{recu3}):
$$\begin{array}{l}
P_{12}=-1,\\
P_{11}=-18x^2+36x-6,\\
\vdots\\
P_0=-11280372719616x^{24}+\cdots-\frac{19737958095}{4},\\
P_{-1}=473775654223872x^{26}+\cdots+85760073216 \textit{  non-identically zero!}
\end{array}$$ and for instance the Galois Group is not the icosahedral group. In conclusion, the Galois group of the equation \eqref{paheq8} is the connected and unsolvable group $\mathrm{SL}(2,\mathbb{C})$.

In the same way we can consider the case when $c^2>4b$, that is, $\rho(t)=A_2\sinh\omega_1t+B_2\cosh\omega_2t$. Owing to the change of variable $t\mapsto it$ transforms trigonometric functions into hyperbolic ones and after similar change of variables as considered before, we arrive again to a non-integrable Heun equation. In general, for $c^2\neq 4b$ which corresponds to $\rho(t)=A_3e^{\omega_1t}+B_3e^{\omega_2t}$ we arrive to a non-integrable confluent Heun equation in where not all singularity is of regular type. We can use the fact that the identity connected component of the Galois group is preserved under algebrization process (see \cite{acthesis,acmowe}), thus, in our case, the Galois group of the normal variational equation will be the same no matter the way in where we algebrize the differential equation.\\

In this paper we shown that all the singularities of the normal variational equation \eqref{paheq8} are of regular type (Heun equation), as well we proven that such differential equation has not Liouvillian solutions and also we obtained that

\begin{align*}
& \omega_1^2= \displaystyle{\frac{1}{2} \left((1+b)+\sqrt{(1-b)^2+c^2}\; \right)},  \\
\mbox{\rm and} \qquad & \omega_2^2= \displaystyle{\frac{1}{2} \left((1+b)-\sqrt{(1-b)^2+c^2}\; \right)},
\end{align*}
being $$\displaystyle{ b=\frac{\lambda}{k\ell^2}},\quad c=\displaystyle{\frac{\epsilon}{k\ell}},\quad f=\left(\frac{\omega_p}{\omega_s}\right)^2.$$
Thus, assuming $f\omega_1\omega_2\neq 0$  for any $B\in \mathbb{C}^*$, we have proven the following result.

\begin{theorem}\label{wilbinteg} Let the parameters of
the Wilberforce-spring pendulum satisfy $c^2-4b\neq0$, then the system is not integrable through meromorphic first
integrals.
\end{theorem}
\begin{proof}
Observe that $\omega_1^2$ is always positive. On the other hand $\omega_2^2=0$ if and only if
$c^2=4b$.
\end{proof}

\appendix
\section{Differential Galois theory}
The Galois theory of differential equations, also called Differential
Galois Theory has been developed by Picard, Vessiot, Kolchin and
currently by a lot of researchers. In particular, we focus in the
Galois theory of linear differential equations, also known as
Picard-Vessiot theory. Following \cite{acmowe} and also
\cite{acthesis}, we present here an algebraic model for functions and
the corresponding Galois theory.
\\

\textbf{Differential Fields.}
Let $K$ be a commutative field of characteristic zero. A {\em
derivation} of $K$ is a map
$\partial_x : K\rightarrow K$
satisfying $\partial_x (a + b) =
\partial_x a + \partial_x b$ and $\partial_x(ab) = \partial_x a \cdot b + a
\cdot\partial_x b$ for all $a,b\in K$.  We then say that
$(K,\partial_{x})$ (or just $K$, when there is no ambiguity) is
a {\it{differential field}} with the derivation $\partial_x$.\\
We assume that $K$ contains an element $x$ such that $\partial_{x}(x)=1$.
Let $\mathcal C$ denote
the field of constants of $K$: $$\mathcal C = \{c\in K |
\partial_x c = 0\}.$$ It is also of characteristic zero and will be
assumed to be algebraically closed. The \emph{coefficient field} for a
differential equation is defined as the smallest differential field
containing all the coefficients
of the equation. \\

Due to we will mostly analyze  second order linear homogeneous
differential equations,
i.e equations of the form
\begin{displaymath}\label{soldeq}
\mathcal L:= \partial^2_xy+a\partial_x y+by=0,\quad a,b\in K.
\end{displaymath}
so the rest of the theory will be explained in this context.
\\

\textbf{Picard-Vessiot Extension.}
Let $L$ be a differential field containing $K$ (a differential
extension of $K$).
We say that $L$ is a \emph{Picard-Vessiot} extension of $K$ for $\mathcal L$
if there exist two linearly independent $y_1, y_2\in L$ solutions of
$\mathcal L$ such
that $L= K\langle y_1, y_2 \rangle$ (i.e $L=K(y_1, y_2,\partial_xy_1,
\partial_xy_2)$)
and $L$ and $K$ has the same field of constants $\mathcal{C}$.\\

In what follows, we choose a Picard-Vessiot extension and the term
``solution of $\mathcal L$'' will mean
``solution of $\mathcal L$ in $L$''. So any solution of $\mathcal L$
is a linear combination
(over $\mathcal{C}$) of $y_{1}$ and $y_{2}$.\\

\textbf{Differential Galois Groups}
A $K$-automorphism $\sigma$ of the Picard-Vessiot extension $L$ is
called a differential automorphism
if it leaves $K$ fixed and commutes with the derivation.  This means
that $\sigma(\partial_xa)=\partial_x(\sigma(a))$ for all $a\in L$ and
$\forall a\in K,$ $\sigma(a)=a$.
\\
The group of all differential automorphisms of $L$
over $K$ is called the {\it{differential Galois group}} of $L$
over $K$ and is denoted by ${\rm DGal}(L/K)$. \\

Given $\sigma \in \mathrm{DGal}(L/K)$, we see that $\{\sigma y_1,
\sigma y_2\}$ are also solutions of $\mathcal L$. Hence there exists a
matrix

$$A_\sigma=
\begin{pmatrix}
a & b\\
c & d
\end{pmatrix}
\in \mathrm{GL}(2,\mathbb{C}),$$ such that
$$\sigma(
\begin{pmatrix}
y_{1} &
y_{2}
\end{pmatrix})
=
\begin{pmatrix}
\sigma (y_{1})     &
\sigma (y_{2})
\end{pmatrix}
=\begin{pmatrix} y_{1}& y_{2}
\end{pmatrix}A_\sigma.$$

As $\sigma$ commutes with the derivation, this extends naturally to an
action on a fundamental solution matrix of the companion first order
system associated with
$\mathcal{L}$.
$$\sigma\left(
\begin{pmatrix}
y_{1}&y_2\\
\partial_xy_1&\partial_xy_{2}
\end{pmatrix}\right)
=
\begin{pmatrix}
\sigma (y_{1})&\sigma (y_2)\\
\sigma (\partial_xy_1)&\sigma (\partial_xy_{2})
\end{pmatrix}
=\begin{pmatrix} y_{1}& y_{2}\\\partial_xy_1&\partial_xy_2
\end{pmatrix}A_\sigma.$$

This defines a faithful representation $\mathrm{DGal}(L/K)\to
\mathrm{GL}(2,\mathbb{C})$ and it is possible to consider
$\mathrm{DGal}(L/K)$ as a subgroup of $\mathrm{GL}(2,\mathbb{C})$.
It depends on the choice of the fundamental system $\{y_1,y_2\}$,
but only up to conjugacy.\\

Recall that an algebraic group $G$ is an algebraic manifold
endowed with a group structure.
Let $\mathrm{GL}(n,\mathbb{C})$ denote, as usual, the set of
invertible $n\times n$  matrices
with entries in $\mathbb{C}$ (and $\mathrm{SL}(n,\mathbb{C})$ be the
set of matrices with determinant
equal to $1$).
A linear algebraic group will be a subgroup of
$\mathrm{GL}(n,\mathbb{C})$ equipped with a
structure of algebraic group.
One of the fundamental results of the Picard-Vessiot theory is the
following theorem.\\

\emph{The differential Galois group $\mathrm{DGal}(L/K)$ is an
algebraic subgroup of $\mathrm{GL}(2,\mathbb{C})$.}\\
In fact, the differential Galois group measures the algebraic
relations between the solutions (and their derivatives).
It is sometimes viewed as the object which should tell ``what algebra
sees of the dynamics of the solutions''.

In an algebraic group $G$, the largest connected algebraic
subgroup of $G$ containing the identity, noted $G^{\circ}$, is a
normal subgroup of finite index.
It is often called the connected component of the identity.
If $G=G^0$ then $G$ is a \textit{connected group}. \\
When $G^0$ satisfies some property,     we say that $G$ virtually
satisfies this property. For example, virtually solvability of $G$
means solvability of $G^0$ and virtual
abelianity of $G$ means abelianity of $G^0$..\\

\textbf{Lie-Kolchin Theorem.} \emph{Let $G\subseteq
\mathrm{GL}(2,\mathbb{C})$ be a virtually solvable group. Then
$G^0$ is triangularizable, i.e it is conjugate to a subgroup of
upper triangular matrices.}\\

\textbf{Algebraic Subgroups of $SL(2,\mathbb{C})$.} We present below
some examples of subgroups of $\mathrm{SL}(2,\mathbb{C})$.\\

\textbf{Reducible subgroups} These are the groups which leave a
non-trivial subspace of $V$ invariant.
They are classified in two categories.\\
\emph{Diagonal groups: }
the identity group: $\{e\}=\left\{\begin{pmatrix}1&0\\0&1\end{pmatrix}\right\}$,
the $n-$roots: $$\mathbb{G}^{[n]}=\left\{\begin{pmatrix}c&0\\0&c^{-1}\end{pmatrix},\quad
c^n=1\right\},$$
the multiplicative group:
$$\mathbb{G}_m=\left\{\begin{pmatrix}c&0\\0&c^{-1}\end{pmatrix},\quad
c\in\mathbb{C}^*\right\}.$$

\emph{Triangular groups: }
the additive group:
$\mathbb{G}_a=\left\{\begin{pmatrix}1&d\\0&1\end{pmatrix},\quad
d\in\mathbb{C}\right\}$,
the $n-$quasi-roots:
$\mathbb{G}^{\{n\}}=\left\{\begin{pmatrix}c&d\\0&c^{-1}\end{pmatrix},\quad
c^n=1, \quad d\in\mathbb{C}\right\}$,
the Borel group:
$$\mathbb{B}=\mathbb{C}^*\ltimes  \mathbb{C}=\left\{\begin{pmatrix}c&d\\0&c^{-1}\end{pmatrix},\quad
c\in\mathbb{C}^*, \quad d\in\mathbb{C}\right\}.$$

\textbf{Irreducible subgroups}
The infinite dihedral group (also called meta-abelian group):\\
$\mathbb{D}_{\infty}=\left\{\begin{pmatrix}c&0\\0&c^{-1}\end{pmatrix},\quad
c\in\mathbb{C}^*\right\}\cup\left\{\begin{pmatrix}0&d\\-d^{-1}&0\end{pmatrix},\quad
d\in\mathbb{C}^*\right\}$ and its finite subgroups $\mathbb{D}_{2n}$
(where $c$  and $d$ spans the $n$-th roots of unity). \\

There are also three other finite irreducible (primitive) groups: the
tetrahedral group $A_4^{\mathrm{SL}_2}$ of order $24$, the octahedral
group $S_4^{\mathrm{SL}_2}$ of order $48$, and the icosahedral group
$A_5^{\mathrm{SL}_2}$ of order $120$.\\

\textbf{Integrability.}
We say that the linear differential equation $\mathcal L$ is
(Liouville) \textit{integrable} if
the Picard-Vessiot extension $L\supset K$ is obtained as a tower
of differential fields $K=L_0\subset L_1\subset\cdots\subset
L_m=L$ such that $L_i=L_{i-1}(\eta)$ for $i=1,\ldots,m$, where
either
\begin{enumerate}
\item $\eta$ is {\emph{algebraic}} over $L_{i-1}$, that is $\eta$ satisfies a
polynomial equation with coefficients in $L_{i-1}$.
\item $\eta$ is {\emph{primitive}} over $L_{i-1}$, that is
$\partial_x\eta \in L_{i-1}$.
\item $\eta$ is {\emph{exponential}} over $L_{i-1}$, that is
$\partial_x\eta /\eta \in L_{i-1}$.
\end{enumerate}
\medskip

We remark that the usual terminology in differential algebra for
integrable equations is that the corresponding Picard-Vessiot
extensions are called \textit{Liouvillian}. The following theorem is
due to Kolchin.\\

\emph{The equation $\mathcal L$ is integrable if and only if
$\mathrm{DGal}(L/K)$ is virtually solvable.}\\

\section{Kovacic algorithm}

Kovacic in 1986 (see \cite{ko}) introduced an algorithm to solve
the differential equation $\partial_x^2\zeta=r\zeta$, where $r\in
\mathbb{C}(x)$.
\\

Each case in Kovacic's algorithm is related with each one of the
algebraic subgroups of ${\rm SL}(2,\mathbb{C})$ and the associated
Riccatti equation
$$\partial_xv=r-v^{2}=\left( \sqrt{r}-v\right)
\left(  \sqrt{r}+v\right),\quad v={\partial_x\zeta\over \zeta}.$$

There are four cases in Kovacic's algorithm. Only for cases 1, 2
and 3 we can solve the differential equation, but for the case 4
the differential equation is not integrable. It is possible that
Kovacic's algorithm can provide us only one solution ($\zeta_1$),
so that we can obtain the second solution ($\zeta_2$) through
\begin{equation}\label{second}
\zeta_2=\zeta_1\int\frac{dx}{\zeta_1^2}.
\end{equation}

For the differential equation given by
$$\partial_x^2\zeta=r\zeta,\qquad r={s\over t},\quad s,t\in \mathbb{C}[x],$$
we use the following notations.
\begin{enumerate}
\item Denote by $\Gamma'$ be
the
set of (finite) poles of $r$, $\Gamma^{\prime}=\left\{  c\in\mathbb{C}%
:t(c)=0\right\}$.

\item Denote by
$\Gamma=\Gamma^{\prime}\cup\{\infty\}$.
\item By the order of $r$ at
$c\in \Gamma'$, $\circ(r_c)$, we mean the multiplicity of $c$ as a
pole of $r$.

\item By the order of $r$ at $\infty$, $\circ\left(
r_{\infty}\right) ,$ we mean the order of $\infty$ as a zero of
$r$. That is $\circ\left( r_{\infty }\right)
=\mathrm{deg}(t)-\mathrm{deg}(s)$.

\end{enumerate}

{\it Case 1.} In this case $\left[ \sqrt{r}\right] _{c}$ and
$\left[ \sqrt{r}\right] _{\infty}$ means the Laurent series of
$\sqrt{r}$ at $c$ and the Laurent series of $\sqrt{r}$ at $\infty$
respectively. Furthermore, we define $\varepsilon(p)$ as follows:
if $p\in\Gamma,$ then $\varepsilon\left( p\right) \in\{+,-\}.$
Finally, the complex numbers $\alpha_{c}^{+},\alpha_{c}^{-},\alpha_{\infty}%
^{+},\alpha_{\infty}^{-}$ will be defined in the first step. If
the differential equation has no poles it only can fall in this
case.
\medskip

{\it Step 1.} Search for each $c \in \Gamma'$ and for $\infty$ the
corresponding situation as follows:

\medskip

\begin{description}

\item[$(c_{0})$] If $\circ\left(  r_{c}\right)  =0$, then
$$\left[ \sqrt {r}\right] _{c}=0,\quad\alpha_{c}^{\pm}=0.$$

\item[$(c_{1})$] If $\circ\left(  r_{c}\right)  =1$, then
$$\left[ \sqrt {r}\right] _{c}=0,\quad\alpha_{c}^{\pm}=1.$$

\item[$(c_{2})$] If $\circ\left(  r_{c}\right)  =2,$ and $$r= \cdots
+ b(x-c)^{-2}+\cdots,\quad \textrm{then}$$
$$\left[ \sqrt {r}\right]_{c}=0,\quad
\alpha_{c}^{\pm}=\frac{1\pm\sqrt{1+4b}}{2}.$$

\item[$(c_{3})$] If $\circ\left(  r_{c}\right)  =2v\geq4$, and $$r=
(a\left( x-c\right)  ^{-v}+...+d\left( x-c\right)
^{-2})^{2}+b(x-c)^{-(v+1)}+\cdots,\quad \textrm{then}$$ $$\left[
\sqrt {r}\right] _{c}=a\left( x-c\right) ^{-v}+...+d\left(
x-c\right) ^{-2},\quad\alpha_{c}^{\pm}=\frac{1}{2}\left(
\pm\frac{b}{a}+v\right).$$

\item[$(\infty_{1})$] If $\circ\left(  r_{\infty}\right)  >2$, then
$$\left[\sqrt{r}\right]
_{\infty}=0,\quad\alpha_{\infty}^{+}=0,\quad\alpha_{\infty}^{-}=1.$$

\item[$(\infty_{2})$] If $\circ\left(  r_{\infty}\right)  =2,$ and
$r= \cdots + bx^{2}+\cdots$, then $$\left[
\sqrt{r}\right]  _{\infty}=0,\quad\alpha_{\infty}^{\pm}=\frac{1\pm\sqrt{1+4b}%
}{2}.$$

\item[$(\infty_{3})$] If $\circ\left(  r_{\infty}\right) =-2v\leq0$,
and
$$r=\left( ax^{v}+...+d\right)  ^{2}+ bx^{v-1}+\cdots,\quad \textrm{then}$$
$$\left[  \sqrt{r}\right]  _{\infty}=ax^{v}+...+d,\quad
and\quad \alpha_{\infty}^{\pm }=\frac{1}{2}\left(
\pm\frac{b}{a}-v\right).$$
\end{description}
\medskip

{\it Step 2.} Find $D\neq\emptyset$ defined by
$$D=\left\{
n\in\mathbb{Z}_{+}:n=\alpha_{\infty}^{\varepsilon
(\infty)}-%
{\displaystyle\sum\limits_{c\in\Gamma^{\prime}}}
\alpha_{c}^{\varepsilon(c)},\forall\left(  \varepsilon\left(
p\right) \right)  _{p\in\Gamma}\right\}  .$$ If $D=\emptyset$,
then we should start with the case 2. Now, if
$\mathrm{Card}(D)>0$, then for each $n\in D$ we search $\omega$
$\in\mathbb{C}(x)$ such that
$$\omega=\varepsilon\left(
\infty\right)  \left[  \sqrt{r}\right]  _{\infty}+%
{\displaystyle\sum\limits_{c\in\Gamma^{\prime}}}
\left(  \varepsilon\left(  c\right)  \left[  \sqrt{r}\right]  _{c}%
+{\alpha_{c}^{\varepsilon(c)}}{(x-c)^{-1}}\right).$$
\medskip

{\it Step 3}. For each $n\in D$, search for a monic polynomial
$P_n$ of degree $n$ with
\begin{equation}\label{recu1}
\partial_x^2P_n + 2\omega \partial_xP_n + (\partial_x\omega + \omega^2
- r) P_n = 0.
\end{equation}
If success is achieved then $\zeta_1=P_n e^{\int\omega}$ is a
solution of the differential equation.  Else, case 1 cannot hold.
\bigskip

{\it Case 2.}  Search for each $c \in \Gamma'$ and for $\infty$
the corresponding situation as follows:
\medskip

{\it Step 1.} Search for each $c\in\Gamma^{\prime}$ and $\infty$
the sets $E_{c}\neq\emptyset$ and $E_{\infty}\neq\emptyset.$ For
each $c\in\Gamma^{\prime}$ and for $\infty$ we define
$E_{c}\subset\mathbb{Z}$ and $E_{\infty}\subset\mathbb{Z}$ as
follows:
\medskip

\begin{description}
\item[($c_1$)] If $\circ\left(  r_{c}\right)=1$, then $E_{c}=\{4\}$

\item[($c_2$)] If $\circ\left(  r_{c}\right)  =2,$ and $r= \cdots +
b(x-c)^{-2}+\cdots ,\ $ then $$E_{c}=\left\{
2+k\sqrt{1+4b}:k=0,\pm2\right\}\cap\mathbb{Z}.$$

\item[($c_3$)] If $\circ\left(  r_{c}\right)  =v>2$, then $E_{c}=\{v\}$

\item[$(\infty_{1})$] If $\circ\left(  r_{\infty}\right)  >2$, then
$E_{\infty }=\{0,2,4\}$

\item[$(\infty_{2})$] If $\circ\left(  r_{\infty}\right)  =2,$ and
$r= \cdots + bx^{2}+\cdots$, then $$E_{\infty }=\left\{
2+k\sqrt{1+4b}:k=0,\pm2\right\}\cap\mathbb{Z}.$$

\item[$(\infty_{3})$] If $\circ\left(  r_{\infty}\right)  =v<2$,
then $E_{\infty }=\{v\}$
\medskip
\end{description}

{\it Step 2.} Find $D\neq\emptyset$ defined by
$$D=\left\{
n\in\mathbb{Z}_{+}:\quad n=\frac{1}{2}\left(  e_{\infty}-
{\displaystyle\sum\limits_{c\in\Gamma^{\prime}}} e_{c}\right)
,\forall e_{p}\in E_{p},\quad p\in\Gamma\right\}.$$ If
$D=\emptyset,$ then we should start the case 3. Now, if
$\mathrm{Card}(D)>0,$ then for each $n\in D$ we search a rational
function $\theta$ defined by
$$\theta=\frac{1}{2}
{\displaystyle\sum\limits_{c\in\Gamma^{\prime}}}
\frac{e_{c}}{x-c}.$$
\medskip

{\it Step 3.} For each $n\in D,$ search a monic polynomial $P_n$
of degree $n$, such that {\small{
\begin{equation}\label{recu2}
\partial_x^3P_n+3\theta
\partial_x^2P_n+(3\partial_x\theta+3\theta
^{2}-4r)\partial_xP_n+\left(
\partial_x^2\theta+3\theta\partial_x\theta
+\theta^{3}-4r\theta-2\partial_xr\right)P_n=0.
\end{equation}}}
If $P_n$ does not
exist, then case 2 cannot hold. If such a polynomial is found, set
$\phi = \theta + \partial_xP_n/P_n$ and let $\omega$ be a solution
of
$$\omega^2 + \phi \omega + {1\over2}\left(\partial_x\phi + \phi^2 -2r\right)=
0.$$

Then $\zeta_1 = e^{\int\omega}$ is a solution of the differential
equation.
\bigskip

{\it Case 3.} Search for each $c \in \Gamma'$ and for $\infty$ the
corresponding situation as follows:
\medskip

{\it Step 1.} Search for each $c\in\Gamma^{\prime}$ and $\infty$
the sets $E_{c}\neq\emptyset$ and $E_{\infty}\neq\emptyset.$ For
each $c\in\Gamma^{\prime}$ and for $\infty$ we define
$E_{c}\subset\mathbb{Z}$ and $E_{\infty}\subset\mathbb{Z}$ as
follows:
\medskip

\begin{description}

\item[$(c_{1})$] If $\circ\left(  r_{c}\right)  =1$, then
$E_{c}=\{12\}$

\item[$(c_{2})$] If $\circ\left(  r_{c}\right)  =2,$ and $r= \cdots +
b(x-c)^{-2}+\cdots$, then
\begin{displaymath}
E_{c}=\left\{ 6+{12k\over m}\sqrt{1+4b}:\quad
k=0,\pm1,\ldots, \pm\frac{m}2,\quad
m=4,6,12\right\}\cap \mathbb{Z}.
\end{displaymath}

\item[$(\infty)$] If $\circ\left(  r_{\infty}\right)  =v\geq2,$ and $r=
\cdots + bx^{2}+\cdots$, then {\small $$E_{\infty }=\left\{
6+{12k\over m}\sqrt{1+4b}:\quad
k=0,\pm1,\ldots, \pm\frac{m}2,\quad
m=4,6,12\right\}\cap\mathbb{Z}.$$}
\medskip
\end{description}

{\it Step 2.} Find $D\neq\emptyset$ defined by
$$D=\left\{
n\in\mathbb{Z}_{+}:\quad n=\frac{m}{12}\left(
e_{\infty}-{\displaystyle\sum\limits_{c\in\Gamma^{\prime}}}
e_{c}\right)  ,\forall e_{p}\in E_{p},\quad p\in\Gamma\right\}.$$
In this case we start with $m=4$ to obtain the solution,
afterwards $m=6$ and finally $m=12$. If $D=\emptyset$, then the
differential equation is not integrable because it falls in the
case 4. Now, if $\mathrm{Card}(D)>0,$ then for each $n\in D$ with
its respective $m$, search a rational function
$$\theta={m\over 12}{\displaystyle\sum\limits_{c\in\Gamma^{\prime}}}
\frac{e_{c}}{x-c}$$ and a polynomial $S$ defined as $$S=
{\displaystyle\prod\limits_{c\in\Gamma^{\prime}}} (x-c).$$

{\it Step 3}. Search for each $n\in D$, with its respective $m$, a
monic polynomial $P_n=P$ of degree $n,$ such that its coefficients
can be determined recursively by
\begin{equation}\label{recu3}
\begin{array}{l}
P_{m}=-P,\quad  \textit{ and for  } i\in\{m,m-1,\ldots,1,0\},\\
P_{i-1}=-S\partial_xP_{i}-\left( \left( m-i\right)
\partial_xS-S\theta\right)  P_{i}-\left( m-i\right)  \left(
i+1\right)  S^{2}rP_{i+1},
\end{array}
\end{equation} and for $i=0$ the polynomial $P_{-1}$ should be identically zero, i.e., $P_{-1}\equiv 0$. If
$P$ does not exist ($P_{-1}$ is not identically zero), then the differential equation is not
integrable because it falls in Case 4. Now, if $P$ exists search
$\omega$ such that $$ {\displaystyle\sum\limits_{i=0}^{m}}
\frac{S^{i}P}{\left( m-i\right)  !}\omega^{i}=0,$$ then a solution
of the differential equation is given by $$\zeta=e^{\int
\omega},$$ where $\omega$ is solution of the previous polynomial
of degree $m$.

\section{Hamiltonian algebrization}
In this section we follow \cite{acthesis, acmowe}. We recall that
there are a lot of differential equations with coefficients that are
not rational functions. For these differential
equations it is useful, when is possible, to replace it by a new
differential equation over the Riemann sphere $\mathbb{P}^1$ (that
is, with rational coefficients). To do this, we can use a change
of variables. The equation over $\mathbb{P}^1$ is called the
\emph{algebraic form} or \emph{algebrization} of the original
equation.\\

\textbf{Hamiltonian change of variable.} A change of
variable $z=z(x)$ is called \textit{Hamiltonian} if
$(z(x),\partial_xz(x))$ is a solution curve of the autonomous
classical one degree of freedom Hamiltonian system
$$\begin{array}{l}\partial_xz=\partial_wH\\
\partial_xw=-\partial_zH\end{array}\quad \textrm{with}\quad
H=H(z,w)={w^2\over 2}+V(z),$$
for some $V\in K$, where $K$ is a differential field. Thus, $z=z(x)$
is a Hamiltonian
change of variable if there exists $\alpha\in K$ such that
$(\partial_xz)^2=\alpha(z)$. More specifically, if $z=z(x)$ is a
Hamiltonian change of variable, we can write
$\partial_xz=\sqrt{\alpha}$, which leads us to the following
notation: $\widehat{\partial}_z=\sqrt{\alpha}\partial_z$.

We can see that $\widehat{\partial}_z$ is a \textit{derivation}
because satisfy
$\widehat{\partial}_z(f+g)=\widehat{\partial}_zf+\widehat{\partial}_zg$
and the Leibnitz rules $$\widehat{\partial}_z(f\cdot
g)=\widehat{\partial}_zf\cdot g +f\cdot\widehat{\partial}_zg,\quad
\widehat{\partial}_z\left(\frac{f}{g}\right)=\frac{\widehat{\partial}_zf\cdot
g -f\cdot\widehat{\partial}_zg}{g^2}.$$ We can notice that the
chain rule is given by $\widehat{\partial}_z(f\circ
g)=\partial_gf\circ g\widehat{\partial}_z(g)\neq
\widehat{\partial}_gf\circ g\widehat{\partial}_z(g)$. The
iteration of $\widehat{\partial}_z$ is given by
$$\widehat{\partial}_z^0=1,\quad
\widehat{\partial}_z=\sqrt{\alpha}\partial_z,\quad
\widehat{\partial}_z^n=\sqrt{\alpha}\partial_z\widehat{\partial}^{n-1}_z=
\underbrace{\sqrt{\alpha}\partial_z\left(\ldots
\left(\sqrt{\alpha}\partial_z\right)\right)}_{n\text{ times }
\sqrt{\alpha}\partial_z}.$$ We call \textit{Hamiltonian Algebrization}
to the algebrization process obtained by a Hamiltonian change of
variable.\\

\textbf{Hamiltonian Algebrization Theorem \cite{acthesis,acmowe}.}
Consider the systems of linear differential equations
$[A]$ and $[\widehat{A}]$ given respectively by
$$\partial_x\mathbf{Y}=-A\mathbf{Y},\quad
\widehat{\partial}_z\widehat{\mathbf{Y}}=-\widehat{A}\widehat{\mathbf{Y}},\quad
A=[a_{ij}],\quad \widehat{A}=[\widehat{a}_{ij}],\quad
\mathbf{Y}=[y_{i1}], \quad
\widehat{\mathbf{Y}}=[\widehat{y}_{i1}],$$ where $a_{ij}\in
K=\mathbb{C}(z(x),\partial_x(z(x)))$,
$\widehat{a}_{ij}\in\mathbb{C}(z)\subseteq
\widehat{K}=\mathbb{C}(z,\sqrt{\alpha})$, $1\leq i\leq n$, $1\leq
j\leq n$, $a_{ij}(x)=\widehat{a}_{ij}(z(x))$ and
$y_{i1}(x)=y_{i1}(z(x))$. Suppose that $L$ and $\widehat{L}$ are
the Picard-Vessiot extensions of $[A]$ and $[\widehat{A}]$
respectively. If the transformation $\varphi$ is given by
$$\varphi:\begin{array}{l}x\mapsto z\\ a_{ij}\mapsto
\widehat{a}_{ij}\\y_{i1}(x)\mapsto\widehat{y}_{i1}(z(x))\\
\partial_x\mapsto\widehat{\partial}_z\end{array},$$
then the following statements hold:
\begin{enumerate}
\item $K\simeq \widehat{K}$,\quad $(K,\partial_x)\simeq
(\widehat{K},\widehat{\partial}_z).$
\item $\mathrm{DGal}(L/K)\simeq\mathrm{DGal}(\widehat{L}/\widehat{K})\subset
\mathrm{DGal}(\widehat{L}/{\mathbb{C}(z)}).$
\item $(\mathrm{DGal}(L/K))^0\simeq
(\mathrm{DGal}(\widehat{L}/{\mathbb{C}(z)})^0.$
\end{enumerate}

\medskip

\medskip

A natural example of Hamiltonian Algebrization, and for instance of
the introduction of the new derivative $\widehat{\partial}_z$, is the
case of second order linear differential equations. Consider
$\partial_x^2y+a\partial_xy+by=0$, using
$\varphi$ we obtain
$\widehat{\partial}_z^2\widehat{y}+\widehat{a}\widehat{\partial}_z\widehat{y}+\widehat{b}\widehat{y}=0$,
which is equivalent to
\begin{equation}\label{eqalgen}
\alpha\partial_z^2\widehat{y}+\left({\partial_x\alpha\over
2}+{\sqrt{\alpha}\widehat{a}
}\right)\partial_z\widehat{y}+\widehat{b}\widehat{y}=0,
\end{equation} where $y(x)=\widehat{y}(z(x))$, $\widehat{a}(z(x))=a(x)$ and
$\widehat{b}(z(x))=b(x)$.\\

In general, for $y(x)=\widehat{y}(z(x))$, the equation
$F(\partial_x^ny,\ldots,y,x)=0$ with coefficients given by
$a_{i_k}(x)$ is transformed in the equation
$\widehat{F}(\widehat{\partial}_z^n\widehat{y},\ldots,\widehat{y},z)=0$
with coefficients given by $\widehat{a}_{i_k}(z)$, where
$a_{i_k}(x)=\widehat{a}_{i_k}(z(x))$. In particular, for
$\sqrt{\alpha}\in\mathbb{C}(z)$ and $\widehat{a}_{i_k}\in\mathbb{C}(z)$, the equation
$\widehat{F}(\widehat{\partial}_z^n\widehat{y},\ldots,\widehat{y},z)=0$
is the Hamiltonian algebrization of
$F(\partial_x^ny,\ldots,y,x)=0$. Now, if each derivation
$\partial_x$ has order even, with $\alpha(z)$ and $\widehat{a}_{i_k}(z)$
being rational functions, then the equation
$F(\partial_x^ny,\ldots,y,x)=0$ admits an Hamiltonian Algebrization for $z=z(x)$. An example, that illustrate this, is given by the following linear differential equation:
$$\partial_x^{2n}y+a_{n-1}(x)\partial_x^{2n-2}y+\ldots+a_2(x)\partial_x^4y+a_1(x)\partial_x^2y+a_0(x)y=0.$$

\textbf{Hamiltonian Algebrization Algorithm Theorem
\cite{acthesis,acmowe}.} In general is very difficult to find a
suitable Hamiltonian change of variable, for this reason Hamiltonian
Algebrization is a method or procedure (not an algorithm!). For
specific families of differential equations we can obtain algorithms
to apply Hamiltonian Algebrization, for example \emph{any differential
equation
$$F(\partial_x^ny,\partial_x^{n-1}y,\ldots,\partial_xy,y;
e^{\lambda_1 x},\ldots,e^{\lambda_k x} )=0,\quad
\mu=p_i\lambda_i,\quad p_i\in\mathbb{Z},\quad 1\leq i\leq k,$$ admits
Hamiltonian Algebrization if and only if $${\lambda_i\over
\lambda_j}\in \mathbb{Q^*},\quad 1\leq i\leq k,\quad 1\leq j\leq k$$
and by means of the Hamiltonian change of variable $z=e^{\mu
x}$ we obtain the equation
$$\widehat{F}(\widehat{\partial}_z^n\widehat{y},\widehat{\partial}_z^{n-1}\widehat{y},\ldots,\widehat{\partial}_z\widehat{y},y;
z)=0$$ with coefficients in $\mathbb{C}(z)$.}

\section*{Acknowledgements}
The first author was partially support by Marie Curie Fellowship Cofund UNITE during his stay at Technical University of Madrid and after by Universidad del Norte. First and third authors are partially supported by the MICIIN/FEDER grant number MTM2009-06973 and by the
 Generalitat de Catalunya grant number 2009SGR859.
 All the authors acknowledge to the anonimous referees by their useful comments and suggestions.

\medskip
Received xxxx 20xx; revised xxxx 20xx.
\medskip

\end{document}